\documentclass[11pt,letterpaper,reqno]{amsart}
\usepackage{graphicx, psfrag, amscd, amssymb, young}
\usepackage{multirow}

\newtheorem{lem}{Lemma}[section]
\newtheorem{thm}[lem]{Theorem}

\newtheorem{cor}[lem]{Corollary}
\newtheorem{exa}[lem]{Example}

\newcommand{\C}{{\mathcal{C}}}

\renewcommand{\H}{{\mathcal{H}}}

\newcommand{\inv}{{\textsf{inv}}}
\newcommand{\wt}{{\textsf{wt}}}

\renewcommand{\P}{{\mathcal{P}}}

\newcommand{\F}{{\mathcal{F}}}
\newcommand{\R}{{\mathcal{R}}}

\newcommand{\U}{{\mathcal{U}}}

\newcommand{\bangle}{\atopwithdelims \langle \rangle}

\title[On $xD$-generalizations of Stirling Numbers and Lah numbers]{On $xD$-generalizations of Stirling Numbers and\\ Lah numbers via Graphs and Rooks}

\author{Sen-Peng Eu}
\address{Department of Mathematics, National Taiwan Normal University, Taipei 106, Taiwan, ROC}
\email{speu@nuk.edu.tw}

\author{Tung-Shan Fu}
\address{Department of Applied Mathematics, National Pingtung University, Pingtung 900, Taiwan, ROC}
\email{tsfu@mail.nptu.edu.tw}

\author{Yu-Chang Liang}
\address{Department of Applied Mathematics, National Pingtung University, Pingtung 900, Taiwan, ROC}
\email{chase2369216@hotmail.com}

\author{Tsai-Lien Wong}
\address{Department of Applied Mathematics, National Sun Yat-sen University, Kaohsiung 804, Taiwan, ROC}
\email{tlwong@math.nsysu.edu.tw}

\begin{document}

\maketitle

\begin{abstract}
This paper studies the generalizations of the Stirling numbers of both kinds and the Lah numbers in association with the normal order problem in the Weyl algebra $W=\langle x,D|Dx-xD=1\rangle$. Any word $\omega\in W$ with $m$ $x$'s and $n$ $D$'s can be expressed in the normally ordered form $\omega=x^{m-n}\sum_{k\ge 0} {{\omega}\brace {k}} x^{k}D^{k}$, where ${{\omega}\brace {k}}$ is known as the Stirling number of the second kind for the word $\omega$. This study considers the expansions of restricted words $\omega$ in $W$ over the sequences $\{(xD)^{k}\}_{k\ge 0}$ and $\{xD^{k}x^{k-1}\}_{k\ge 0}$. Interestingly, the coefficients in individual expansions turn out to be generalizations of the Stirling numbers of the first kind and the Lah numbers. The coefficients will be determined through enumerations of some combinatorial structures linked to the words $\omega$, involving decreasing forest decompositions of quasi-threshold graphs and non-attacking rook placements on Ferrers boards.  Extended to $q$-analogues, weighted refinements of the combinatorial interpretations are also investigated for words in the $q$-deformed Weyl algebra.
\end{abstract}

\section{Introduction}
The Stirling numbers of both kinds and the Lah numbers are ubiquitous in combinatorics. In this paper, we study the generalizations of these numbers in association with the normal order problem in the Weyl algebra $W$ generated by two operators $x$ and $D$ with the relation $Dx-xD=1$.
A well known example of $W$ is the algebra of differential operators applied to polynomials $f(x)$, where the operator $x$ acts as multiplication by $x$, and $D$ as differentiation with respect to $x$, i.e., $(Df)(x)=\frac{d}{dx}f(x)$. Clearly, $(Dx-xD)f(x)=f(x)$. Any word $\omega\in W$ can be expressed in the normally ordered form
\[
\omega=\sum_{i,j\ge 0} c_{ij}x^{i}D^{j}
\]
for some non-negative integers $c_{ij}$. The problem of finding explicit formula for the normal order coefficients $c_{ij}$ appears in the theory of quantum mechanics, where the symbols $x$ and $D$ act as the boson annihilation operator and creation operator, denoted as $a$ and $a^{\dag}$, satisfying the commutation relation $aa^{\dag}-a^{\dag}a=1$.

\subsection{Stirling numbers of the second kind}
For the word $\omega=(xD)^n$, it has long been obtained by Scherk \cite{Scherk} in 1823 that the normal order coefficients of $(xD)^n$ are the \emph{Stirling numbers of the second kind}, denoted as ${{n}\brace {k}}$, i.e.,
\begin{equation} \label{eq:XD-2nd-Stirling}
(xD)^n=\sum_{k=0}^n {{n}\brace {k}} x^{k}D^{k}.
\end{equation}
These numbers ${{n}\brace {k}}$ count the number of ways to partition the set $[n]:=\{1,2,\dots,n\}$ into $k$ non-empty subsets. Generally speaking, any word $\omega$ in the Weyl algebra $W$ with $m$ $x$'s and $n$ $D$'s can be uniquely expanded over the sequence $\{x^{k}D^{k}\}_{k\ge 0}$ as
\begin{equation} \label{eq:word-2nd-Stirling}
\omega=x^{m-n} \sum_{k\ge 0} {{\omega}\brace k} x^{k}D^{k}.
\end{equation}
The integer sequence $({{\omega}\brace k})_{k\ge 0}$ are called the \emph{Stirling numbers of the second kind for the words} $\omega$.  There are a lot of studies on the normal order coefficients ${{\omega}\brace k}$ for various words $\omega$ in $W$. Specifically, we focus on the combinatorial interpretations of ${{\omega}\brace k}$ involving independent set decompositions of quasi-threshold graphs in \cite{EGH} and rook placements on Ferrers boards in \cite{Navon,Varvak}.

Navon \cite{Navon} associated $\omega$ with a Ferrers board within the rectangle in the plane $\mathbb{Z}\times\mathbb{Z}$ with the lower-left corner $(0,0)$ and the upper-right corner $(m,n)$ and gave a combinatorial interpretation of ${{\omega}\brace k}$ in terms of (non-attacking) rook placements on the board. Varvak demonstrated this interpretation and obtained a $q$-analogous result \cite[Theorems 3.2 and 6.3]{Varvak}. Recently, Engbers, Galvin and Hilyard \cite{EGH} studied the numbers ${{\omega}\brace k}$ on a collection of restricted words $\omega\in W$. A word $\omega\in W$ with $n$ $x$'s and $n$ $D$'s is called a \emph{Dyck word} of semi-length $n$ if every prefix of $\omega$ has at least as many $x$'s as $D$'s.
The word $\omega$ is associated with a quasi-threshold graph $G_{\omega}$ (defined in next section) and the number ${{\omega}\brace k}$ is realized as the number of ways to partition the graph $G_{\omega}$ into $k$ non-empty independent sets  \cite[Theorem 2.3]{EGH}.

\subsection{Stirling numbers of the first kind and Lah numbers for words} The motivation of this study comes from the following Stirling inversion,  mentioned in \cite[Exercise 1.46]{EC1}.
\begin{equation} \label{eq:XD-1st-Stirling}
x^{n}D^{n}=\sum_{k=0}^n (-1)^{n-k} {{n}\brack {k}} (xD)^k,
\end{equation}
where ${{n}\brack {k}}$ is the \emph{Stirling number of the first kind}.
Among other combinatorial interpretations, ${{n}\brack {k}}$ counts the number of ways to partition the complete graph on vertices $[n]$ into $k$-component \emph{decreasing forests} \cite[A008275]{OEIS}. By a \emph{decreasing tree} we mean an unordered rooted tree in which every path from the root is decreasing.

It turns out that the Dyck words in $W$ can also be expanded uniquely over the sequence $\{(xD)^{k}\}_{k\ge 0}$.
For a Dyck word $\omega$ with $n$ $x$'s and $n$ $D$'s, we propose the \emph{Stirling numbers of the first kind for the word} $\omega$, denoted as ${{\omega}\brack {k}}$, defined by the following expansion
\begin{equation} \label{eq:word-1st-Stirling}
\omega=\sum_{k=0}^n (-1)^{n-k}{{\omega}\brack k} (xD)^{k}.
\end{equation}
Note that the normal order coefficients of $\omega$ can be obtained by applying the transform in Eq.\,(\ref{eq:XD-2nd-Stirling}) to the expansion in Eq.\,(\ref{eq:word-1st-Stirling}).

Closed to the Stirling numbers of both kinds, the (unsigned) \emph{Lah numbers}, denoted as ${{n}\bangle {k}}$, are the connecting constants of the polynomial identity
\[
x(x+1)\cdots (x+n-1)=\sum_{k=0}^n {{n}\bangle {k}} x(x-1)\cdots (x-k+1),
\]
which yields
\[
{{n}\bangle {k}}=\sum_{j=k}^n {{n}\brack {j}} {{j}\brace {k}}.
\]
One of the combinatorial interpretations of ${{n}\bangle {k}}$ is the number of ways to partition the complete graph on vertices $[n]$ into a disjoint union of $k$ decreasing forests.
For combinatorial interest, we derive an identity
\begin{equation} \label{eq:XD-Lah-number}
x^{n}D^{n}=\sum_{k=0}^n (-1)^{n-k} {{n}\bangle {k}} xD^{k}x^{k-1},
\end{equation}
linking the word $x^{n}D^{n}$ to the sequence $\{xD^{k}x^{k-1}\}_{k\ge 0}$ by the Lah numbers.

For any word $\omega\in W$ with $n$ $x$'s and $n$ $D$'s, starting with an $x$, we propose the \emph{Lah numbers for the word} $\omega$, denoted as ${{\omega}\bangle {k}}$, defined by the following expansion
\begin{equation} \label{eq:word-Lah-number}
\omega=\sum_{k=0}^n (-1)^{n-k}{{\omega}\bangle k} xD^{k}x^{k-1}.
\end{equation}
Note that the normal order coefficients of $\omega$ can be obtained by applying the following transform to Eq.\,(\ref{eq:word-Lah-number})
\begin{equation} \label{eq:transform}
xD^{n}x^{n-1}=\sum_{k=0}^n {{n}\bangle {k}} x^{k}D^{k}.
\end{equation}
For convenience, sometimes we call ${{\omega}\brace k}, {{\omega}\brack k}$ the $xD$-\emph{Stirling numbers} and call ${{\omega}\bangle k}$ the $xD$-\emph{Lah numbers}. One of our main purposes is to give combinatorial interpretations of ${{\omega}\brack {k}}$ and ${{\omega}\bangle {k}}$ for Dyck words $\omega$ in terms of decreasing forest decompositions of the quasi-threshold graphs $G_{\omega}$ (Theorem \ref{thm:meaning-Stir-1st-Dyck} and Theorem \ref{thm:meaning-Lah-Dyck}) and in terms of rook placements on Ferrers boards (Corollary \ref{cor:rook-xD-Stirling-1st} and Corollary \ref{cor:rook-xD-Lah}).



\subsection{$q$-analogues of Stirling numbers and Lah numbers}
We shall extend the combinatorial interpretations of ${{\omega}\brack k}$ and ${{\omega}\bangle k}$ in the context of the $q$-\emph{deformed Weyl algebra} $W$ of operators $x$ and $D$ with the relation $Dx-qxD=1$ ($q$ denotes an indeterminate).

The problem of normal ordering in the $q$-deformed Weyl algebra $W$ has been studied by Katriel \cite{Katriel,KK} and Schork \cite{Schork}. For any word $\omega$ in $W$ with $m$ $x$'s and $n$ $D$'s, a $q$-analogue of the $xD$-Stirling number of the second kind, denoted as ${{\omega}\brace {k}}_q$, is defined by the following expansion
\begin{equation} \label{eq:q-xD-2nd-Stirling}
\omega=x^{m-n} \sum_{k\ge 0} {{\omega}\brace k}_q x^{k}D^{k}.
\end{equation}
Varvak  \cite{Varvak} gave a combinatorial interpretation for ${{\omega}\brace {k}}_q$ by defining an inversion statistic for the rook placements on the Ferrers board associated with $\omega$. For Dyck words $\omega$, Engbers et al. gave a combinatorial interpretation for ${{\omega}\brace {k}}_q$, which is quite involved, by defining a weight function for the partitions of the associated graph $G_{\omega}$ into $k$ non-empty independent sets \cite[Theorem 2.12]{EGH}.

Extended to $q$-analogues, for a Dyck word $\omega\in W$ with $n$ $x$'s and $n$ $D$'s, we define the $q$-analogue of the $xD$-Stirling number of the first kind, denoted as ${{\omega}\brack {k}}_q$, by the expansion
\begin{equation} \label{eq:q-word-1st-Stirling}
\omega=\sum_{k=0}^n (-1)^{n-k}{{\omega}\brack k}_q (xD)^{k}.
\end{equation}

For a Dyck word $\omega\in W$ with $n$ $x$'s and $n$ $D$'s, starting with an $x$, we define the $q$-analogue of the $xD$-Lah number, denoted as ${{\omega}\bangle {k}}_q$, by the expansion
\begin{equation} \label{eq:q-word-Lah-number}
\omega=\sum_{k=0}^n (-1)^{n-k}{{\omega}\bangle k}_q xD^{k}x^{k-1}.
\end{equation}
Our second set of main results are weighted realizations of ${{\omega}\brack k}_q$ (Theorems \ref{thm:weighted-q-xD-Stirling-1st-kind} and \ref{thm:Ferrers-meaning-q-Stirling-1st}) and ${{\omega}\bangle k}_q$ (Theorem \ref{thm:Ferrers-meaning-q-Lah-number}).

Meanwhile, considering the expansion of the specific word $\omega=(xD)^n$ in Eq.\,(\ref{eq:q-word-Lah-number}), we present a new $q$-Stirling number of the second kind (Theorem \ref{thm:meaning-q-Stirling-2nd}), which is different from the one introduced by Carlitz \cite{Carlitz}.
Moreover, considering the expansion of the word $\omega=x^nD^n$ in Eq.\,(\ref{eq:q-word-Lah-number}), we present a new $q$-Lah number, ${{n}\bangle {k}}_q$, realized by weighted decreasing forest decompositions of a complete graph (Theorem \ref{thm:weighted-q-Lah-number}). We remark that this is different from the $q$-Lah numbers of Garsia and Remmel \cite{GR} and the $q$-Lah numbers defined by Lindsay, Mansour and Shattuck in \cite{LMS}.

\subsection{Rook factorization theorem and chromatic polynomials}
Regarding rook replacements on Ferrers boards, a prominent result in rook theory is the Rook Factorization Theorem, given by Goldman, Joichi and White \cite{GJW}, which states that a factorial rook polynomial can be completely factorized into linear factors. There is also a $q$-counting rook configuration result given by Garsia and Remmel \cite{GR}.

Varvak \cite{Varvak} demonstrated that the $xD$-Stirling number of the second kind, ${{\omega}\brace {k}}$, and its $q$-analogue can be evaluated by the factorial rook polynomials.
Making use of Varvak's method, we derive the following identities for evaluating of the numbers ${{\omega}\brack {k}}$ and ${{\omega}\bangle {k}}$
\begin{equation} \label{eq:evaluation-xD-rook-factors}
\sum_{k=0}^n (-1)^{n-k} {{\omega}\brack {k}} z^k=\prod_{i=1}^n (z-c_i+i)=\sum_{k=0}^n (-1)^{n-k} {{\omega}\bangle {k}} z(z+1)\cdots(z+k-1),
\end{equation}
where $c_1,\dots,c_n$ are the column-heights of the Ferrers board associated with the word $\omega$. Their $q$-analogous results are also obtained (Theorems \ref{thm:rook-factorization-Stirling-1st}-\ref{thm:rook-factorization-q-Lah-number}).

In particular, for Dyck words $\omega$, the generating function for the (signed) numbers ${{\omega}\brack {k}}$ in Eq.\,(\ref{eq:evaluation-xD-rook-factors}) has an equivalent description in terms of the chromatic polynomials of the associated quasi-threshold graph $G_{\omega}$.
By Whitney's theorem \cite{Whitney}, we have another interpretation for ${{\omega}\brack {k}}$, counting the number of subgraphs consisting of $n-k$ edges of $G_{\omega}$ without broken circuits. We also present a bijection between the decreasing forest decompositions of $G_{\omega}$ and the broken-circuit free subgraphs of $G_{\omega}$ (Theorem \ref{thm:broken-circuit-bijection}).

The rest of the paper is organized as follows. In Section 2 we shall give combinatorial interpretations of the $xD$-Stirling number of the first kind ${{\omega}\brack {k}}$ and the $xD$-Lah number ${{\omega}\bangle {k}}$ in terms of decreasing forest decompositions of quasi-threshold graphs. In Section 3 we shall give a $q$-analogous result for the $xD$-Stirling number of the first kind, as well as a new $q$-Stirling number of the second kind and a new $q$-Lah number. In Section 4 we turn to rook placements on Ferrers boards and give combinatorial interpretations of  the numbers ${{\omega}\brack {k}}_q$ and ${{\omega}\bangle {k}}_q$. Section 5 will be devoted to the rook factorization results for ${{\omega}\brack {k}}$  and ${{\omega}\bangle {k}}$. In Section 6 we describe the chromatic polynomials of the quasi-threshold graph $G_{\omega}$ and the bijective result.

\section{Stirling numbers of the 1st kind and Lah numbers for Dyck words}  In this section, we explore combinatorial interpretations of ${{\omega}\brack {k}}$ and ${{\omega}\bangle {k}}$ for Dyck words $\omega$ in terms of graph decompositions.

Let $\C_n$ denote the set of Dyck words of semi-length $n$. A Dyck word $\omega\in\C_n$ is visualized with a lattice path from $(0,0)$ to $(n,n)$ in the plane $\mathbb{Z}\times\mathbb{Z}$, taking $x$ as the {\em north} step $(0,1)$ and $D$ as the {\em east} step $(1,0)$, that stays weakly above the line $y=x$, called a \emph{Dyck path} of length $n$. We shall use Dyck words and Dyck paths interchangeably.
Respecting the first east step returning to the line $y=x$, we factorize $\omega$ as $\omega=x\omega'D\omega''$, called the \emph{standard factorization} of $\omega$,
 where $\omega'$ and $\omega''$ are Dyck paths (possibly empty). We call the prefix $\mu=x\omega'D$ the first \emph{block} of $\omega$.

Engbers et al. \cite{EGH} associated $\omega$ with a graph $G_{\omega}$. The construction is described below.
The east steps $D$'s of $\omega$ are labeled $1,2,\dots, n$ from left to right. The north steps $x$'s of $\omega$ are matched up with $D$'s that face each other, in the sense that the line segment (also called a \emph{tunnel}) from the midpoint of a north step to the midpoint of an east step has slope 1 and stays below the path.  Each matched pair $(x,D)$ will be converted into a vertex.  On the vertices $[n]$, the graph $G_{\omega}$ is constructed inductively as follows.
\begin{enumerate}
\item If $\omega$ is empty then $G_{\omega}$ is empty.
\item Otherwise, factorize $\omega$ in the standard form $\omega=x\omega' D\omega''$. Then the graph $G_{\omega}$ is the disjoint union of $G_{\omega'}+K_1$ and $G_{\omega''}$, where $G_{\omega'}+K_1$ is the graph obtained from $G_{\omega'}$ by adding a dominating vertex with the label of $D$.
\end{enumerate}
The graph $G_{\omega}$ is also known as a \emph{quasi-threshold} graph. For example, the graph $G_{\omega}$ shown in Figure \ref{fig:word-graph} is associated with the Dyck word $\omega=xxDxxDDD$.

\begin{figure}[ht]
\begin{center}
\psfrag{w}[][][0.85]{$\omega$}
\psfrag{G_w}[][][0.85]{$G_{\omega}$}
\includegraphics[width=2.5in]{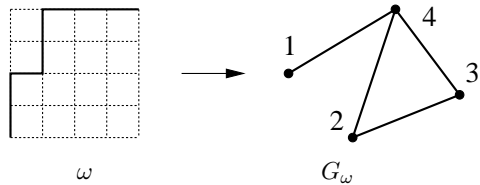}
\end{center}
\caption{\small The quasi-threshold graph $G_{\omega}$ associated with $\omega=xxDxxDDD$.}
\label{fig:word-graph}
\end{figure}

\subsection{The $xD$-Stirling numbers of the first kind.}
Recall that an unordered rooted tree $T$ on the vertex set $[n]$ is {\em
decreasing} if every path from the root is decreasing. The order of the children of a vertex is irrelevant.
A \emph{decreasing forest} $F$ on $[n]$ is a forest such that every component is a decreasing tree.
For any Dyck word $\omega\in\C_n$, we shall prove that the $xD$-Stirling number ${{\omega}\brack {k}}$
coincides with the number of ways to partition the graph $G_{\omega}$ into $k$-component decreasing forests.

The proof proceeds by induction on the semi-length $n$ of $\omega$, with the initial conditions ${{\omega}\brack {0}}=\delta_{n,0}$ for $n\ge 0$ and ${{\omega}\brack {k}}=0$ for $0\le n< k$.

\begin{lem} \label{lem:XD-Stirling-1st-kind-recurrence} Given a Dyck word $\omega\in\C_n$ with the standard factorization $\omega=x\omega'D\omega''$, let $m$ be the semi-length of the first block $\mu=x\omega'D$. Then the following relations hold.
\begin{enumerate}
\item For $1\le k\le n$, we have
\[
{{\omega}\brack {k}} = \sum_{k_1=1}^{m} {{\mu}\brack {k_1}}{{\omega''}\brack {k-k_1}}.
\]
\item For the first block $\mu=x\omega'D$ and $1\le k_1\le m$, we have
\[
{{\mu}\brack {k_1}}=\sum_{\ell=k_1-1}^{m-1} {{\omega'}\brack {\ell}}{{\ell}\choose {k_1-1}}.
\]
\end{enumerate}
\end{lem}

\begin{proof}
(i) By Eq.\,(\ref{eq:word-1st-Stirling}), we observe that
\[
\omega= \mu\omega'' = \left(\sum_{j=0}^{m} (-1)^{m-j} {{\mu}\brack {j}} (xD)^{j} \right)\left(\sum_{i=0}^{n-m} (-1)^{n-m-i} {{\omega''}\brack {i}} (xD)^{i} \right).
\]
Extracting the coefficient of $(xD)^{k}$ on both sides, the assertion follows.

(ii) Making reduction with the relation $xD=Dx-1$, we have
\begin{align*}
\mu= x\omega' D &= x\left(\sum_{\ell=0}^{m-1} (-1)^{m-1-\ell} {{\omega'}\brack {\ell}} (xD)^{\ell} \right)D \\
      &= x\left(\sum_{\ell=0}^{m-1} (-1)^{m-1-\ell} {{\omega'}\brack {\ell}}(Dx-1)^{\ell} \right)D \\
      &= x\left(\sum_{\ell=0}^{m-1} (-1)^{m-1-\ell} {{\omega'}\brack {\ell}} \sum_{i=0}^{\ell} {{\ell}\choose {i}} (-1)^{\ell-i}(Dx)^{i}\right)D \\
      &= \sum_{\ell=0}^{m-1} (-1)^{m-1-\ell} {{\omega'}\brack {\ell}} \sum_{i=0}^{\ell} {{\ell}\choose {i}} (-1)^{\ell-i}(xD)^{i+1}.
\end{align*}
Extracting the coefficient of $(xD)^{k_1}$ on both sides, the assertion follows.
\end{proof}

\smallskip
Now, we give a combinatorial interpretation of ${{\omega}\brack {k}}$ for Dyck words $\omega\in\C_n$.
Let $\F({\omega},k)$ be the collection of partitions of the graph $G_{\omega}$ into $k$-component decreasing forests. We assume $|\F({\omega},0)|=\delta_{n,0}$ for $n\ge 0$ and $|\F({\omega},k)|=0$ for $0\le n<k$.

\begin{thm} \label{thm:meaning-Stir-1st-Dyck} For any word $\omega\in\C_n$ and $1\le k\le n$, we have
\[
|\F({\omega},k)|={{\omega}\brack {k}}.
\]
\end{thm}

\begin{proof}
In the standard factorization $\omega=x\omega'D\omega''$, let $m$ be the semi-length of the first block $\mu=x\omega'D$.
Note that the graph $G_{\omega}$ is the disjoint union of $G_{\mu}$ of $G_{\omega''}$. Any forest
$\gamma\in\F(\omega,k)$ is a disjoint union of a member $\alpha \in\F(\mu,k_1)$ and $\beta \in\F(\omega'',k-k_1)$ for some $k_1$ ($1\le k_1\le m$). Hence $|\F({\omega},k)|$ satisfies the relation $|\F({\omega},k)| = \sum_{k_1=1}^{m} |\F({\mu},k_1)|\cdot |\F({\omega''},k-k_1)|$.

For the first block $\mu=x\omega'D$,
the graph $G_{\omega'}$ is obtained from $G_{\mu}$ by removing the dominating vertex $m$.
For any forest $\alpha\in\F(\mu,k_1)$, removing the vertex $m$ from $\alpha$ leads to a forest $\alpha\cap G_{\omega'}\in\F(\omega',\ell)$
 for some $\ell$ ($k_1-1\le \ell\le m-1$).
Moreover, the forest $\alpha$ can be constructed from a forest $\beta\in\F({\omega'},\ell)$
by joining $\ell-k_1+1$ components of $\beta$ to the vertex $m$.
Since there are ${{\ell}\choose {\ell-k_1+1}}={{\ell}\choose {k_1-1}}$
ways to choose  $\ell-k_1+1$ components from $\beta$, $|\F({\mu},k_1)|$
satisfies the relation $|\F({\mu},k_1)|=\sum_{\ell=k_1-1}^{m-1} |\F({\omega'},\ell)| {{\ell}\choose {k_1-1}}$.

By Lemma \ref{lem:XD-Stirling-1st-kind-recurrence},  the numbers $|F({\omega},k)|$ and ${{\omega}\brack {k}}$
share the same recurrence relations. The assertion follows.
\end{proof}

\begin{exa} {\rm
For the word $\omega=xxDxxDDD$, we have $\omega=-2xD+5(xD)^2-4(xD)^3+(xD)^4$.
The graph $G_{\omega}$ is shown in Figure \ref{fig:word-graph}. For $1\le k\le 4$, the sets $\F(\omega,k)$ of partitions of $G_{\omega}$
into $k$-component decreasing forests are shown in Figure \ref{fig:k-comp-forests}.
}
\end{exa}

\begin{figure}[ht]
\begin{center}
\includegraphics[width=4.4in]{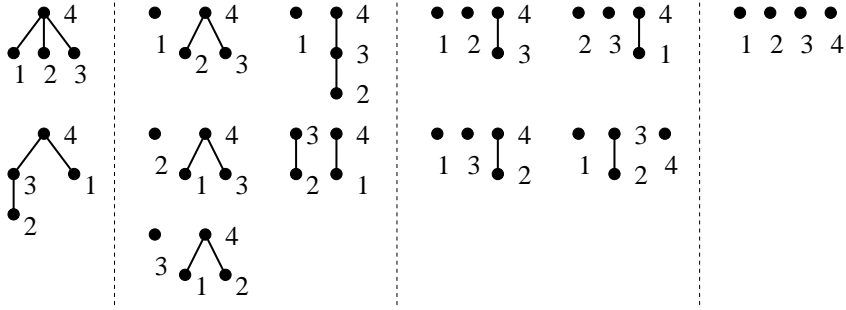}
\end{center}
\caption{\small The members in $\F(\omega,k)$ of the graph associated with the word $\omega=xxDxxDDD$.}
\label{fig:k-comp-forests}
\end{figure}

Setting $\omega=x^{n}D^{n}$ in Theorem \ref{thm:meaning-Stir-1st-Dyck}, the graph $G_{\omega}$ is the complete graph on vertices $[n]$ and hence $|\F(\omega,k)|={{n}\brack {k}}$. This proves the identity in Eq.\,(\ref{eq:XD-1st-Stirling}).

\subsection{The $xD$-Lah numbers.} For any Duck word $\omega\in\C_n$, we shall prove that the $xD$-Lah number ${{\omega}\bangle {k}}$ coincides with the number of ways
to partition the graph $G_{\omega}$ into a disjoint union of $k$ decreasing forests.
The proof is similar to the proof of Theorem \ref{thm:meaning-Lah-Dyck}, with the initial conditions ${{\omega}\bangle {0}}=\delta_{n,0}$ for $n\ge 0$ and ${{\omega}\bangle {k}}=0$ for $0\le n<k$.

The following derivative identity will be used to derive recurrence relation for ${{\omega}\bangle {k}}$.

\begin{lem} \label{lem:derivative} For all $n\ge 1$ and $m\ge 1$, we have
\[
x^{m}D^{n}=\sum_{j\ge 0} (-1)^{j}{{m}\choose {j}}{{n}\choose {j}} j! D^{n-j}x^{m-j}.
\]
\end{lem}

\begin{proof}  For $m=1$, we prove $xD^{n}=D^{n}x-nD^{n-1}$ by induction on $n$. For $n=1$, it is the relation $Dx=xD+1$ of the Weyl algebra. For $n\ge 2$, we observe that
\begin{align*}
xD^n = (xD^{n-1})D &= (D^{n-1}x-(n-1)D^{n-2})D \\
     &= D^{n-1}(Dx-1)-(n-1) D^{n-1}\\
     &= D^{n}x-nD^{n-1},
\end{align*}
as required. Suppose the assertion holds for all $m< k$ and $n\ge 1$. For $m=k$ and $n=1$, the identity $x^{k}D=Dx^{k}-kx^{k-1}$ can be proved in a similar manner as above. For $n\ge2$, we observe that
\begin{align*}
x^kD^n = x(x^{k-1}D^{n}) &= \sum_{j\ge 0} (-1)^{j}{{k-1}\choose {j}}{{n}\choose {j}} j! (xD^{n-j})x^{k-1-j} \\
     &= \sum_{j\ge 0} (-1)^{j}{{k-1}\choose {j}}{{n}\choose {j}} j! (D^{n-j}x^{k-j}-(n-j)D^{n-1-j}x^{k-1-j}).\\
\end{align*}
The coefficient of $D^{n-j}x^{k-j}$ is
\[
(-1)^{j}{{k-1}\choose {j}}{{n}\choose {j}} j!-(-1)^{j-1}{{k-1}\choose {j-1}}{{n}\choose {j-1}} (j-1)! (n-j+1)=(-1)^{j}{{k}\choose {j}}{{n}\choose {j}} j!,
\]
as required.
\end{proof}


\begin{lem} \label{lem:XD-Lah-number-recurrence} Given a Dyck word $\omega\in\C_n$ with a standard factorization $\omega=x\omega'D\omega''$, let $m$ be the semi-length of the first block $\mu=x\omega'D$. Then the following relations hold.
\begin{enumerate}
\item For $1\le k\le n$, we have
\[
{{\omega}\bangle {k}} = \sum_{k_1=0}^{m}\sum_{k_2=0}^{n-m} {{\mu}\bangle {k_1}}{{\omega''}\bangle {k_2}}
 {{k_1}\choose {k_1+k_2-k}}{{k_2}\choose {k_1+k_2-k}}(k_1+k_2-k)!.
\]
\item For the first block $\mu=x\omega'D$ and $1\le k_1\le m$, we have
\[
{{\mu}\bangle {k_1}}={{\omega'}\bangle {k_1-1}}+2k_1{{\omega'}\bangle {k_1}}+(k_1+k_1^2){{\omega'}\bangle {k_1+1}}.
\]
\end{enumerate}
\end{lem}

\begin{proof}
(i) By Eq.\,(\ref{eq:word-Lah-number}), we observe that
\begin{align}
\mu\omega'' &= \left(\sum_{k_1=0}^{m} (-1)^{m-k_1} {{\mu}\bangle {k_1}} xD^{k_1}x^{k_1-1} \right)\left(\sum_{k_2=0}^{n-m} (-1)^{n-m-k_2} {{\omega''}\bangle {k_2}} xD^{k_2}x^{k_2-1} \right) \\
     &= \sum_{k_1=0}^{m}\sum_{k_2=0}^{n-m} (-1)^{n-k_1-k_2} {{\mu}\bangle {k_1}}{{\omega''}\bangle {k_2}}  \label{eq:Leibniz-rule} xD^{k_1}x^{k_1}D^{k_2}x^{k_2-1}.
\end{align}
By Lemma \ref{lem:derivative}, we have
\[
x^{k_1}D^{k_2} =\sum_{j\ge 0} (-1)^j {{k_1}\choose {j}} {{k_2}\choose {j}} {j!} D^{k_2-j}x^{k_1-j}.
\]
Substituting back to Eq.\,(\ref{eq:Leibniz-rule}) and extracting the coefficient of $xD^{k}x^{k-1}$ on both sides, we have
\[
{{\omega}\bangle {k}} =  \sum_{k_1=0}^{m}\sum_{k_2=0}^{n-m} {{\mu}\bangle {k_1}}{{\omega''}\bangle {k_2}} {{k_1}\choose {k_1+k_2-k}}{{k_2}\choose {k_1+k_2-k}}(k_1+k_2-k)!.\\
\]

(ii) Making use of the identities in Lemma \ref{lem:derivative}, we observe that
\begin{align*}
\mu= x\omega' D &= x\left(\sum_{\ell=0}^{m-1} (-1)^{m-1-\ell} {{\omega'}\bangle {\ell}} xD^{\ell}x^{\ell-1} \right)D \\
      &= x\left(\sum_{\ell=0}^{m-1} (-1)^{m-1-\ell}  {{\omega'}\bangle {\ell}} (D^{\ell}x-\ell D^{\ell-1})\big(Dx^{\ell-1}-(\ell-1)x^{\ell-2}\big)\right) \\
      &= \sum_{\ell=0}^{m-1} (-1)^{m-1-\ell} {{\omega'}\bangle {\ell}}\big(xD^{\ell+1}x^{\ell}-2\ell xD^{\ell}x^{\ell-1}+\ell(\ell-1)xD^{\ell-1}x^{\ell-2}\big).
\end{align*}
Extracting the coefficient of $(xD)^{k_1}$ on both sides, the assertion follows.
\end{proof}

\smallskip
Now, we give a combinatorial interpretation of ${{\omega}\bangle {k}}$ for Dyck words $\omega$.
Let $\H({\omega},k)$ be the collection of partitions of $G_{\omega}$ into a disjoint union of $k$ decreasing forests. We assume $|\H({\omega},0)|=\delta_{n,0}$ for $n\ge 0$ and $|\H({\omega},k)|=0$ for $0\le n<k$.

\begin{thm} \label{thm:meaning-Lah-Dyck} For any word $\omega\in\C_n$ and $1\le k\le n$, we have
\[
|\H({\omega},k)|={{\omega}\bangle {k}}.
\]
\end{thm}

\begin{proof} In the standard factorization $\omega=x\omega'D\omega''$, let $m$ be the semi-length of the first block $\mu=x\omega'D$. We shall prove that the cardinality of $\H({\omega},k)$ satisfies the following relations.
\begin{enumerate}
\item For $1\le k\le n$, we have
\[
|\H({\omega},k)| = \sum_{k_1=0}^{m}\sum_{k_2=0}^{n-m} |\H({\mu},k_1)|\cdot |\H({\omega''},k-k_1)|{{k_1}\choose {k_1+k_2-k}}{{k_2}\choose {k_1+k_2-k}}(k_1+k_2-k)!.
\]
\item For the first block $\mu=x\omega'D$ and $1\le k_1\le m$, we have
\[
|\H({\mu},k_1)|= |\H({\omega'},k_1-1)|+2k_1|\H({\omega'},k_1)|+ (k_1+k_1^2)|\H({\omega'},k_1+1)|.
\]
\end{enumerate}
Note that $G_{\omega}$ is a disjoint union of $G_{\mu}$ and $G_{\omega''}$. Any forest $\gamma\in\H({\omega},k)$ can be constructed from a member $\alpha\in\H({\mu},k_1)$ and a member $\beta\in\H({\omega''},k_2)$ for some $k_1,k_2$ with $k_1+k_2\ge k$ such that $\gamma$ consists of the forests from the following categories.
\begin{itemize}
\item Choose $k_1+k_2-k$ forests from $\alpha$ and choose $k_1+k_2-k$ forests from $\beta$. Use one-to-one correspondence to merge the two families of forests into $k_1+k_2-k$ forests.
\item The remaining $k-k_2$ forests of $\alpha$.
\item The remaining $k-k_1$ forests of $\beta$.
\end{itemize}
The right-hand side of the equation in (i) is exactly the possibilities of $\gamma\in\H({\omega},k)$.

For the first block $\mu=x\omega' D$, the graph $G_{\omega'}$ is obtained from $G_{\mu}$ by removing the dominating vertex $m$.  For any forest $\alpha\in\H({\mu},k_1)$, removing the vertex $m$ from $\alpha$ leads to a forest $\alpha\cap G_{\omega'}\in\H({\omega'},\ell)$ for some $\ell\in\{k_1-1,k_1,k_1+1\}$. Moreover, the forest $\alpha$ can be constructed from a forest $\beta\in\H({\omega'},\ell)$ according to the following cases.
\begin{itemize}
\item $\ell=k_1-1$. The forest $\alpha$ is obtained from $\beta$ by adding the $k_1$th forest, consisting of the vertex $m$.
\item $\ell=k_1$. The forest $\alpha$ is obtained from $\beta$ by adding the vertex $m$ as a trivial tree to one of the $k_1$ forests of $\beta$.
\item $\ell=k_1$. Choose one of the $k_1$ forests of $\beta$, say $F$. The forest $\alpha$ is  obtained from $\beta$ by joining all of the components of $F$ to the vertex $m$.
\item $\ell=k_1+1$. Choose one of the $k_1+1$ forests of $\beta$, say $F$, and turn $F$ into a tree $T$ by joining all of the components of $F$ to the vertex $m$. The forest $\alpha$ is obtained from $\beta$ by adding $T$ to one of the remaining $k_1$ forests of $\beta$.
\end{itemize}
The right-hand side of the equation in (ii) is exactly the possibilities of $\alpha\in\H({\mu},k_1)$.
\end{proof}

\begin{exa} {\rm
For the word $\omega=xxDxxDDD$, we have $\omega=-12xD+24(xD^2x)-10(xD^3x^2)+(xD^4x^3)$. The 24 ways to partition $G_{\omega}$ into a disjoint union of 2 deceasing forests are shown in Figure \ref{fig:24-disj-2-forests}.
}
\end{exa}

\begin{figure}[ht]
\begin{center}
\includegraphics[width=4.8in]{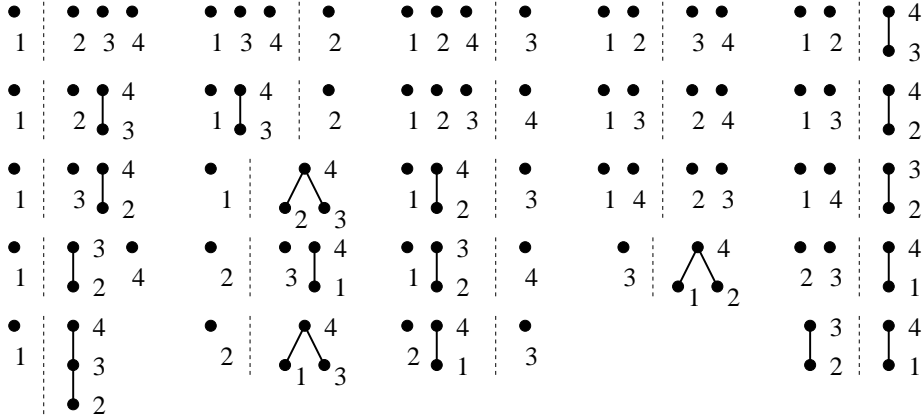}
\end{center}
\caption{\small The partitions of the graph $G_{\omega}$ into 2 decreasing forests for $\omega=xxDxxDDD$.}
\label{fig:24-disj-2-forests}
\end{figure}

Setting $\omega=x^{n}D^{n}$ in Theorem \ref{thm:meaning-Lah-Dyck}, the graph $G_{\omega}$ is the complete graph on vertices $[n]$ and hence $|\H(\omega,k)|={{n}\bangle {k}}$. This proves the identity in Eq.\,(\ref{eq:XD-Lah-number}).

\section{On $q$-analogues of Stirling numbers and Lah numbers}

\subsection{A $q$-analogue of the $xD$-Stirling number of the 1st kind}
Recall that for a Dyck word $\omega\in\C_n$ in the $q$-deformed Weyl algebra $W$, the $q$-analogue of the $xD$-Stirling number of the first kind, ${{\omega}\brack {k}}_q$, is defined by the expansion in Eq.\,(\ref{eq:q-word-1st-Stirling}).
With the standard factorization $\omega=x\omega'D\omega''$ of $\omega$, let $m$ be the semi-length of the first block $\mu=x\omega'D$. Making use of the relation $xD=q^{-1}(Dx-1)$ and the same argument as in the proof of Lemma \ref{lem:XD-Stirling-1st-kind-recurrence}, it is straightforward to derive the following relations, with the initial conditions ${{\omega}\brack {0}}_q=\delta_{n,0}$ for $n\ge 0$ and ${{\omega}\brack {k}}_q=0$ for $0\le n< k$.

\begin{enumerate}
\item For $1\le k\le n$, we have
\[
{{\omega}\brack {k}}_q = \sum_{k_1=1}^{m} {{\mu}\brack {k_1}}_q {{\omega'}\brack {k-k_1}}_q.
\]
\item For the first block $\mu=x\omega'D$ and $1\le k_1\le m$, we have
\[
{{\mu}\brack {k_1}}_q=\sum_{\ell=k_1-1}^{m-1} q^{-\ell}{{\omega'}\brack {\ell}}_q {{\ell}\choose {k_1-1}}.
\]
\end{enumerate}

In the following, we present a combinatorial interpretation of ${{\omega}\brack {\ell}}_q$ by defining a weight function for the forests in $\F({\omega},k)$.

We write a decreasing forest $F$ in a \emph{canonical form} such that the components are arranged in increasing order of the roots from left to right. Moreover, if a vertex has more than one child then
the children are in increasing order from left to right.
Given a Dyck word $\omega\in\C_n$ with the quasi-threshold graph $G=G_{\omega}$, let $G_{i}$ be the induced subgraph of $G_{\omega}$ on the vertices $\{1,2,\dots,i\}$. Let $Q_i$ be the component of $G_{i}$ containing the vertex $i$ and let $Q^*_i$ be the graph obtained from $Q_i$ by removing the vertex $i$. For a forest $\alpha\in\F(\omega,k)$, let $t_i(\alpha)$ be the number of components in the graph $\alpha\cap Q^*_i$ for $1\le i\le n$, and define the weight $\wt(\alpha)$ of $\alpha$ by
\[
\wt(\alpha):=t_1(\alpha)+t_2(\alpha)+\cdots+t_n(\alpha).
\]
Let $f_q(\omega,k)$ denote the (negative) weight polynomial for $\F({\omega},k)$ defined as
\[
f_q(\omega,k)=\sum_{\alpha\in\F(\omega,k)} q^{-\wt(\alpha)}.
\]

\begin{thm} \label{thm:weighted-q-xD-Stirling-1st-kind} For any word $\omega\in\C_n$ and $1\le k\le n$, we have
\[
f_q(\omega,k)={{\omega}\brack {k}}_q.
\]
\end{thm}

\begin{proof} (i) Since $G_{\omega}$ is a disjoint union of $G_{\mu}$ of $G_{\omega''}$, any decreasing forest $\gamma\in\F({\omega},k)$ is the union of $\gamma\cap G_{\mu} \in\F(\mu,k_1)$ and $\gamma\cap G_{\omega''} \in\F({\omega''},k-k_1)$ for some $k_1$ ($1\le k_1\le m$). Hence
\[
f_q(\omega,k)=\sum_{k_1=1}^{m} f_q(\mu,k)\cdot f_q(\omega'',k).
\]
(ii) Recall that the vertex $m$ is the dominating vertex in $G_{\mu}$. As shown in the proof of Theorem \ref{thm:meaning-Stir-1st-Dyck},
for any forest $\alpha\in\F({\mu},k_1)$, removing the vertex $m$ leads to a forest $\alpha\cap G_{\omega'}\in\F(\omega',\ell)$ for some $\ell$ ($k_1-1\le \ell\le m-1$), in which case the vertex $m$ contributes a weight of $\ell$ to the forest $\alpha$.
Moreover, the forest $\alpha$ can be constructed from a forest $\beta\in\F({\omega'},\ell)$
by joining $\ell-k_1+1$ components of $\beta$ to the vertex $m$.
Hence the polynomial $f_q(\mu,k_1)$ satisfies the relation
\[
f_q({\mu},k_1)=\sum_{\ell=k_1-1}^{m-1} q^{-\ell} f_q({\omega'},\ell){{\ell}\choose {k_1-1}}.
\]
The assertion follows from the observation that the polynomials $f_q({\omega},k)$ and ${{\omega}\brack {k}}_q$ share the same recurrence relation.
\end{proof}

\begin{exa} \label{exa:q-weight-k-forests} {\rm
For the word $\omega=xxDxxDDD$, the coefficients of the expansion
$\omega=\sum_{k=1}^{4} (-1)^{4-k}{{\omega}\brack {k}}_q (xD)^{k}$ are listed in
Table \ref{tab:q-xD-Stirling-1st}. For $1\le k\le 4$, the members in $\F(\omega,k)$,
along with their contributions to the $q$-polynomial $f_q(\omega,k)$, are shown in Figure \ref{fig:weight-k-comp-forests}.
}
\end{exa}

\begin{table}[ht]
\caption{The $q$-analogue of the $xD$-Stirling numbers ${{\omega}\brack {k}}_q$ for $\omega=xxDxxDDD$.}
\centering
\begin{tabular} {c|ccccccc}
\hline
$k$ &  1 & & 2 & & 3 & & 4\\
\hline \\  [-1ex]
$(-1)^{4-k}{{\omega}\brack {k}}_q$ & $-(q^{-4}+q^{-3})$ & &  $3q^{-4}+2q^{-3}$ & & $-(3q^{-4}+q^{-3})$ & & 1 \\  [2ex]
\hline
\end{tabular}
\label{tab:q-xD-Stirling-1st}
\end{table}

\begin{figure}[ht]
\begin{center}
\psfrag{0}[][][0.85]{$q^{0}$}
\psfrag{-3}[][][0.85]{$q^{-3}$}
\psfrag{-4}[][][0.85]{$q^{-4}$}
\includegraphics[width=4.4in]{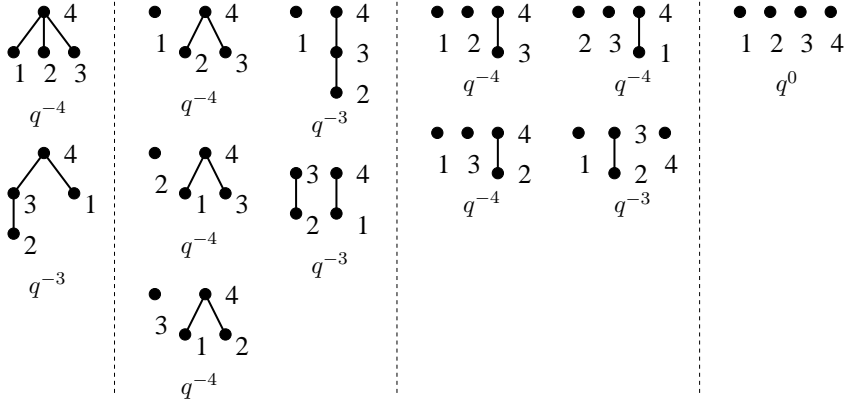}
\end{center}
\caption{\small The members in $\F(\omega,k)$ and their contributions to $f_q(\omega,k)$ for $\omega=xxDxxDDD$.}
\label{fig:weight-k-comp-forests}
\end{figure}

\subsection{Two $q$-Stirling numbers of the 2nd kind} We consider the specific word $\omega=(xD)^n$ in the $q$-deformed Weyl algebra expanding over the sequences  $\{x^{k}D^{k}\}_{k\ge 0}$ and $\{xD^{k}x^{k-1}\}_{k\ge 0}$. We define two $q$-Stirling numbers of the second kind, denoted by ${{n}\brace {k}}_q$ and $\overline{{{n}\brace {k}}}_q$, as the coefficients of the following expansions
\begin{align}
(xD)^{n} &=\sum_{k=0}^n {{n}\brace {k}}_q x^{k}D^{k} \label{eq:q-Stir-2nd-kind}\\
(xD)^{n} &=\sum_{k=0}^n (-1)^{n-k} \overline{{{n}\brace {k}}}_q xD^{k}x^{k-1}. \label{eq:q-Lah-number}
\end{align}
We remark that the former $q$-Stirling number of the second kind ${{n}\brace {k}}_q$ coincides with Carlitz's $q$-Stirling number \cite{Carlitz}, which satisfies the recurrence
\[
{{n}\brace {k}}_q = q^{k-1}{{n-1}\brace {k-1}}_q + [k]_q{{n-1}\brace {k}}_q,
\]
where $[n]_q:=1+q+\cdots+q^{n-1}$ and $[0]_q=1$, with the initial conditions ${n\brace 0}_q=\delta_{n,0}$ for $n\ge 0$ and ${n\brace k}_q=0$ for $0\le n<k$.
Engbers et al. \cite{EGH} gave a combinatorial interpretation of ${{n}\brace {k}}_q$, which is quite involved. As a new generalization, we shall give a combinatorial interpretation for the latter $q$-Stirling number of the second kind $\overline{{{n}\brace {k}}}_q$ (Theorem \ref{thm:meaning-q-Stirling-2nd}).

Making use of the relation $xD=q^{-1}(Dx-1)$, it is straightforward to derive the following identities by the same argument as in the proof of Lemma \ref{lem:derivative}.

\begin{lem} \label{lem:q-xD^n-reduction} For all $n\ge 0$, we have
\begin{enumerate}
\item $xD^{n}=q^{-n}(D^{n}x-[n]_q D^{n-1})$,
\item $x^{n}D=q^{-n}(Dx^{n}-[n]_q x^{n-1})$.
\end{enumerate}
\end{lem}

With the initial conditions $\overline{{n\brace 0}_q}=\delta_{n,0}$ for $n\ge 0$ and $\overline{{n\brace k}_q}=0$ for $0\le n<k$, the polynomial $\overline{{{n}\brace {k}}}_q$ satisfies
the following recurrence relation.

\begin{lem} \label{lem:q-Stirling-2nd-recurrence} For $1\le k\le n$, we have
\[
\overline{{{n}\brace {k}}}_q =\frac{1}{q^{k-1}}\overline{{{n-1}\brace {k-1}}}_q+\frac{[k]_q}{q^k}\overline{{{n-1}\brace {k}}}_q.
\]
\end{lem}

\begin{proof} Making use of the relations in Lemma \ref{lem:q-xD^n-reduction}, we observe that
\begin{align*}
(xD)^n &= (xD)(xD)^{n-1} \\
       &= \sum_{k=0}^{n-1} (-1)^{n-1-k} \overline{{{n-1}\brace {k}}}_q (xD)(xD^{k}x^{k-1}) \\
       &= \sum_{k=0}^{n-1} (-1)^{n-1-k} \overline{{{n-1}\brace {k}}}_q {q^{-k}} (xD)(D^{k}x- [k]_qD^{k-1})x^{k-1}\\
       &= \sum_{k=0}^{n-1} (-1)^{n-1-k} \overline{{{n-1}\brace {k}}}_q {q^{-k}} (xD^{k+1}x^{k}- [k]_q xD^{k}x^{k-1})\\
\end{align*}
Extracting the coefficient of $xD^{k}x^{k-1}$ on both sides, the assertion follows.
\end{proof}

Now, we present a realization of $\overline{{{n}\brace {k}}}_q$.
Let $\P(n,k)$ be the collection of partitions of $[n]$ into $k$ non-empty subsets, called \emph{blocks}. For a partition $\pi\in\P(n,k)$, we arrange the blocks of $\pi$ in a sequence $B_1,B_2,\dots,B_k$ in increasing order of their least elements.  We define the weight $\wt(\pi)$ of the partition $\pi$ by
\[
\wt(\pi):=\sum_{j=1}^k \big(j\cdot |B_j|-1\big).
\]
For example, if $\pi=127|3|489|56\in\P(9,4)$ then $\wt(\pi)=18$.
Let $p_q(n,k)$ denote the (negative) weight polynomial for $\P(n,k)$ defined as
\[
p_q(n,k)=\sum_{\pi\in\P(n,k)} q^{-\wt(\pi)}.
\]

\begin{thm} \label{thm:meaning-q-Stirling-2nd} For $1\le k\le n$, we have
\[
p_q(n,k)=\overline{{{n}\brace {k}}}_q.
\]
\end{thm}

\begin{proof} We shall prove that the polynomial $p_q(n,k)$ satisfies the following relation
\begin{equation} \label{eq:p-q(n,k)-recurrence}
p_q(n,k)=\frac{1}{q^{k-1}} p_q(n-1,k-1)+\frac{[k]_q}{q^k}p_q(n-1,k).
\end{equation}
On the right-hand side of Eq.\,(\ref{eq:p-q(n,k)-recurrence}), we observe that the first term  is the distribution of all members $\pi\in\P(n,k)$ in which the $k$th block consists of the element $n$, contributing a weight of $k-1$ to $\pi$. The second term is the distribution of the members $\pi\in\P(n,k)$ in which the element $n$ occurs in a block with at least one element in $[n-1]$. Note that the element $n$ contributes a weight of $j$ to $\pi$ if $n$ is in the $j$th block for some $j$ ($1\le j\le k$). This proves the recurrence relation Eq.\,(\ref{eq:p-q(n,k)-recurrence}).

By Lemma \ref{lem:q-Stirling-2nd-recurrence}, the polynomials $p_q(n,k)$ and $\overline{{{n}\brace {k}}}_q$ share the same recurrence relation. The assertion follows.
\end{proof}

\begin{exa} {\rm
The coefficients of the expansion $(xD)^4=\sum_{k=1}^{4} (-1)^{4-k} \overline{{{4}\brace {k}}}_q xD^{k}x^{k-1}$  are listed in Table \ref{tab:q-xD-Stirling-2nd}.
The members in $\P(4,2)$, along with their weights are shown in Table \ref{tab:2-partitions-of-4}.
}
\end{exa}

\begin{table}[ht]
\caption{The $q$-Stirling numbers $\overline{{{4}\brace {k}}}_q$ for $1\le k\le 4$.}
\centering
\begin{tabular} {c|ccccccc}
\hline
$k$ &  1 & & 2 & & 3 & & 4\\
\hline \\  [-1ex]
$(-1)^{4-k}\overline{{{4}\brace {k}}}_q$ & $-q^{-3}$ & &  $q^{-5}+3q^{-4}+3q^{-3}$ & & $-(q^{-6}+2q^{-5}+3q^{-4})$ & & $q^{-6}$ \\  [2ex]
\hline
\end{tabular}
\label{tab:q-xD-Stirling-2nd}
\end{table}

\begin{table}[ht]
\caption{The members in $\P(4,2)$ and their weights.}
\centering
\begin{tabular} {c|ccccccccccccc}
\hline
$\pi$ &  $1|234$ & & $134|2$ & & $124|3$ & & $123|4$ & & $12|34$ & & $13|24$ & & $14|23$ \\
\hline \\  [-1ex]
$\wt(\pi)$ & 5 & &  3  & &  3  & &  3  & &  4  & &  4  & & 4 \\
\hline
\end{tabular}
\label{tab:2-partitions-of-4}
\end{table}

\subsection{$q$-Lah number}
We shall present a new $q$-Lah number, ${{n}\bangle {k}}_q$, by the expansion in Eq.\,(\ref{eq:q-word-Lah-number}) of the word $\omega=x^{n}D^n$ in the $q$-deformed Weyl algebra, i.e.,
\[
x^{n}D^{n}=\sum_{k=0}^{n} (-1)^{n-k} {{n}\bangle {k}}_q xD^{k}x^{k-1}.
\]

Making use of the relations in Lemma \ref{lem:q-xD^n-reduction}, it is straightforward to derive the following recurrence in the same manner as the proof of Lemma \ref{lem:XD-Lah-number-recurrence}(ii), with the initial conditions ${n\bangle 0}_q=\delta_{n,0}$ for $n\ge 0$ and ${n\bangle k}_q=0$ for $0\le n<k$.

\begin{lem}
For $1\le k\le n$, we have
\[
{{n}\bangle {k}}_q=\frac{1}{q^{2k-2}} {{n-1}\bangle {k-1}}_q+\frac{(1+q)[k]_q}{q^{2k}}{{n-1}\bangle {k}}_q+\frac{[k]_q[k+1]_q} {q^{2k+1}}{{n-1}\bangle {k+1}}_q.
\]
\end{lem}

In the following, we present a realization of ${{n}\bangle {k}}_q$. Note that the graph associated with the word $\omega=x^{n}D^{n}$ is the complete graph on vertices $[n]$, i.e., $G_{\omega}=K_n$. Let $\H(n,k)$ be the set of partitions of $G_{\omega}$ into a disjoint union of $k$ decreasing forests. We write a forest in the \emph{canonical form} such that its components are arranged in increasing order of their roots.
For a member $\alpha\in\H(n,k)$, we arrange the forests of $\alpha$ in increasing order of their first roots.
For $1\le m\le n$, let $K_m$ denote the complete subgraph of $G_{\omega}$ on the vertices $\{1,2,\dots,m\}$.
For any member $\alpha\in\H(n,k)$, let $F_1,F_2,\dots,F_d$ be the $d$-tuple of forests of $\alpha\cap K_m$ for some integer $d$. Note that the vertex $m$ is the greatest vertex in $\alpha\cap K_m$. Let $T(m)$ denote the component of $\alpha\cap K_m$ rooted at $m$ and let $T^{*}(m)$ denote the forest obtained from $T(m)$ by removing the root $m$. Suppose $T(m)$ is in the forest $F_j$ ($1\le j\le d$). We define two numbers $r_m(\alpha)$ and $s_m(\alpha)$ according to the following cases.
\begin{enumerate}
\item $F_j$ has only one component. Then $j=d$. We assign $r_m(\alpha)=d-1$. Moreover, if $T(m)$ is a single vertex then we assign $s_m(\alpha)=d-1$ otherwise $T^{*}(m)$ is a forest, say the $\ell$th forest, in the graph $\alpha\cap K_{m-1}$ for some $\ell$ ($1\le \ell\le d$) and we assign $s_m(\alpha)=\ell$.
\item $F_j$ has more than one component. Then we assign $r_m(\alpha)=j$. Moreover, if $T(m)$ is a single vertex then we assign $s_m(\alpha)=d$ otherwise $T^{*}(m)$ is a forest, say the $\ell$th forest, in the graph $\alpha\cap K_{m-1}$ for some $\ell$ ($1\le \ell\le d+1$) and we assign $s_m(\alpha)=\ell$.
\end{enumerate}
The weight $\wt(\alpha)$ of $\alpha$ is defined by
\[
\wt(\alpha):=\sum_{m=1}^{n} r_m(\alpha)+s_m(\alpha).
\]
Let $h_q(n,k)$ denote the (negative) weight polynomial for $\H(n,k)$ defined as
\[
h_q(n,k)=\sum_{\alpha\in\H(n,k)} q^{-\wt(\alpha)}.
\]

\begin{thm} \label{thm:weighted-q-Lah-number} For $1\le k\le n$, we have
\[
h_q(n,k)={{n}\bangle {k}}_q.
\]
\end{thm}

\begin{proof} We claim that the polynomial $h_q(n,k)$ satisfies the following relation
\begin{equation*} 
h_q(n,k)=\frac{1}{q^{2k-2}}h_q(n-1,k-1)+\frac{(1+q)[k]_q}{q^{2k}}h_q(n-1,k)+\frac{[k]_q[k+1]_q} {q^{2k+1}} h_q(n-1,k+1).
\end{equation*}
As shown in the proof of Theorem \ref{thm:meaning-Lah-Dyck},
for any member $\alpha\in\H(n,k)$, removing the vertex $n$ from $\alpha$ leads to $\alpha\cap K_{n-1}\in\H(n-1,\ell)$ for some $\ell$ ($k-1\le \ell\le k+1$).  Then $\alpha$ is in one of the following forms.
\begin{itemize}
\item $\ell=k-1$. Then tree $T(n)$, containing a single vertex, forms the last forest of $\alpha$. Then the vertex $n$ contributes the weight of $2k-2$ to $\alpha$.
\item $\ell=k$. The tree $T(n)$ forms the last forest of $\alpha$ and the forest $T^{*}(n)$ is one of the $k$ forests of $\alpha\cap K_{n-1}$, say the $j$th forest. Then the vertex $n$ contributes the weight of $k-1+j$ to $\alpha$.
\item $\ell=k$. The tree $T(n)$, containing a single vertex, is in one of the $k$ forests of $\alpha$, say the $j$th forest. Then the vertex $n$ contributes the weight of $k+j$ to $\alpha$.
\item $\ell=k+1$. The tree $T(n)$ is in one of the $k$ forests of $\alpha$, say the $j$th forest. The forest $T^{*}(n)$ is one of the $k+1$ forests of $\alpha\cap K_{n-1}$, say the $i$th forest. Then the vertex $n$ contributes the weight of $j+i$ to $\alpha$.
\end{itemize}
Hence the polynomial $h_q(n,k)$ satisfies the relation mentioned above.
The assertion follows from the fact that the polynomials $h_q(n,k)$ and ${{n}\bangle {k}}_q$ share the same recurrence relation.

\end{proof}

\begin{exa} {\rm
The coefficients of the expansion $x^4D^4=\sum_{k=1}^{4} (-1)^{4-k}{{4}\bangle {k}}_q xD^{k}x^{k-1}$ are listed in Table \ref{tab:q-Lah-4-k}, where ${{4}\bangle {3}}_q=q^{-7}+2q^{-8}+3q^{-9}+3q^{-10}+2q^{-11}+q^{-12}$. The 12 ways to partition $K_4$ into a disjoin union of 3 forests, along with their contributions to the $q$-polynomial $h_q(4,3)$, are shown in Figure \ref{fig:K4-2-forests}.
}
\end{exa}

\begin{table}[ht]
\caption{The $q$-Lah numbers ${{4}\bangle {k}}_q$ for $1\le k\le 4$.}
\centering
\begin{tabular} {c|ccccccc}
\hline
$k$ &  1 & & 2 & & 3 & & 4\\
\hline \\  [-1ex]
$(-1)^{4-k}{{4}\bangle {k}}_q$ & $-\dfrac{[2]_q[3]_q[4]_q}{q^{9}}$ & &  $\dfrac{[3]_q[3]_q[4]_q}{q^{11}}$  & & $-\dfrac{[3]_q[4]_q}{q^{12}}$
& & $\dfrac{1}{q^{12}}$ \\  [2ex]
\hline
\end{tabular}
\label{tab:q-Lah-4-k}
\end{table}

\begin{figure}[ht]
\begin{center}
\psfrag{7}[][][0.85]{$q^{-7}$}
\psfrag{8}[][][0.85]{$q^{-8}$}
\psfrag{9}[][][0.85]{$q^{-9}$}
\psfrag{10}[][][0.85]{$q^{-10}$}
\psfrag{11}[][][0.85]{$q^{-11}$}
\psfrag{12}[][][0.85]{$q^{-12}$}
\includegraphics[width=4.2in]{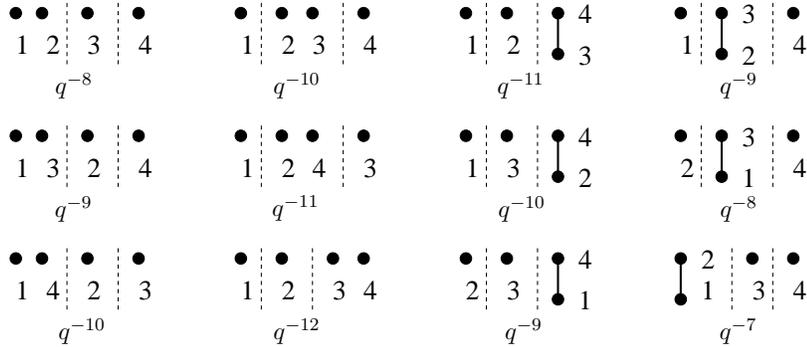}
\end{center}
\caption{\small The 12 members in $\H(4,3)$ along with their contributions to $h_q(4,3)$.}
\label{fig:K4-2-forests}
\end{figure}


\section{Ferrers boards and rook placements} In this section, we present combinatorial interpretations for the numbers ${{\omega}\brack {k}}_q$ and ${{\omega}\bangle {k}}_q$ in terms of rook placements on Ferrers boards.

For a positive integer $n$, consider the $n\times n$ square in the plane $\mathbb{Z}\times\mathbb{Z}$ with the lower-left corner $(0,0)$ and the upper-right corner $(n,n)$. A word $\omega$ in the $q$-deformed Weyl algebra $W$ with $n$ $x$'s and $n$ $D$'s forms a lattice path $\omega$ from $(0,0)$ to $(n,n)$. The region below the path $\omega$ within the $n\times n$ square is called the \emph{Ferrers board} of $\omega$, denoted by $B_{\omega}$. A board consists of an array of cells arranged in rows and columns. The rows (resp. columns) of the board $B_{\omega}$ are indexed $1,2,\dots, n$ from bottom to top (resp. from left to right) and the $(i,j)$ cell is the intersection of the $i$th row and the $j$th column. A consecutive $xD$ steps in the path $\omega$ is called a \emph{peak}. A cell (along the path $\omega$) with a peak on the upper-left corner is called a \emph{peak-cell} of $B_{\omega}$. A \emph{$k$-rook placement} of $B_{\omega}$ is a way to place $k$ non-attacking rooks on the board $B_{\omega}$ (i.e., no two rooks in the same row or column).

\subsection{$q$-Stirling number of the 1st kind for Dyck words} For any Dyck word $\omega\in\C_n$, notice that the board $B_{\omega}$ can always accommodate $n$ non-attacking rooks.
Given a $n$-rook placement of $B_{\omega}$, a rook at the $(i,j)$ cell is \emph{white} if there is no rook placed in the $(a,b)$ cells with $a<i$ and $b>j$, otherwise it is \emph{black}. Namely, there is no rook placed south-east of a white rook. Let $\R(\omega,k)$ be the collection of $n$-rook placements of $B_{\omega}$ with $k$ white rooks. For such a rook placement $\sigma\in \R(\omega,k)$, we define the statistic $\inv(\sigma)$ to be the number of cells in $B_{\omega}$ that either
do not have a rook above them on the same column or to the left of them in the same row, or have a black rook on them. For example, the rook placement shown in Figure \ref{fig:inv-rooks} is a member $\sigma\in\R(\omega,3)$ with $\inv(\sigma)=4$, where the Ferrers board $B_{\omega}$ is associated with the word $\omega=xxDxxDDD$.

\begin{figure}[ht]
\begin{center}
\includegraphics[width=1.6in]{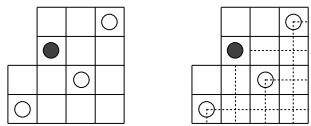}
\end{center}
\caption{\small A rook placement $\sigma\in\R(\omega,3)$ with $\inv(\sigma)=4$ for $\omega=xxDxxDDD$.}
\label{fig:inv-rooks}
\end{figure}

Let $r_q(\omega,k)$ denote the $q$-polynomial of $\R(\omega,k)$ defined as
\[
r_q(\omega,k)=\sum_{\sigma\in \R(\omega,k)} q^{-\inv(\sigma)}.
\]

\begin{thm} \label{thm:Ferrers-meaning-q-Stirling-1st} For any Dyck word $\omega\in\C_n$ in the $q$-deformed Weyl algebra $W$, we have
$r_q(\omega,k)={{\omega}\brack {k}}_q$, i.e.,
\[
\omega=\sum_{k=0}^{n} (-1)^{n-k} r_q(\omega,k) (xD)^{k}.
\]
\end{thm}

\begin{proof} Consider the expansion of $\omega$ over the sequence $\{(xD)^{k}\}_{k\ge 0}$ in Eq.\,(\ref{eq:q-word-1st-Stirling}), the coefficient  ${{\omega}\brack {k}}_q$ is the number of ways to obtain the word $(xD)^k$ from $\omega$, by successively substituting $q^{-1}(Dx-1)$ for $xD$. In terms of Ferrers boards, replacing a peak $xD$ by $Dx$ (resp. by $-1$) is equivalent to deleting that peak-cell (resp. deleting that peak-cell along with its row and column),  both replacements carrying the weight of $q^{-1}$. Hence the coefficient ${{\omega}\brack {k}}_q$ is the number of weighted reductions of the board $B_{\omega}$ to $B_{\mu}$, where $\mu=(xD)^k$. Note that the board $B_{\mu}$ has a unique way to place $k$ non-attacking rooks, i.e., in the $k$ peak-cells along the path $\mu$, which amount to the white rooks of $B_{\omega}$. Moreover, the $n-k$ cells that are deleted along with their rows and columns do not have a row or column in common, which amount to the black rooks of $B_{\omega}$. The sign $(-1)^{n-k}$ is also justified.
Hence in the board $B_{\omega}$, all replacements take place at the cells that either have a black rook on them, or do not have a rook above them on the same column or to the left of them in the same row. Each of these cells is assigned the weight of $q^{-1}$. The weight of
a $n$-rook placement can be considered as the product of the weights of all cells of the board. Then the coefficient  ${{\omega}\brack {k}}_q$ is the weight distribution of the $n$-rook placements of $B_{\omega}$ with $k$ white rooks, which is exactly the $q$-polynomial $r_q(\omega,k)$ of $\R(\omega,k)$.
\end{proof}

\begin{exa} {\rm
As shown in Example \ref{exa:q-weight-k-forests}, for $\omega=xxDxxDDD$, the coefficients of the expansion $\omega=\sum_{k=1}^{4} (-1)^{4-k} r(\omega,k)_q (xD)^{k}$ are listed in Table \ref{tab:q-xD-Stirling-1st}, where
$r_q(\omega,3)=3q^{-4}+q^{-3}$. The 4 members in $\R(\omega,3)$ along with their contributions to $r_q(\omega,3)$ are shown in Figure \ref{fig:white-rooks}.
}
\end{exa}

\begin{figure}[ht]
\begin{center}
\psfrag{-3}[][][0.85]{$q^{-3}$}
\psfrag{-4}[][][0.85]{$q^{-4}$}
\includegraphics[width=3.5in]{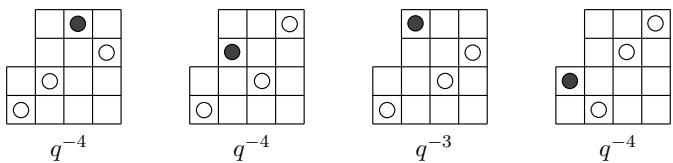}
\end{center}
\caption{\small The 4 members in $\R(\omega,3)$ along with their contributions to $r_q(\omega,3)$.}
\label{fig:white-rooks}
\end{figure}

Setting $q=1$ in Theorem \ref{thm:Ferrers-meaning-q-Stirling-1st}, we get a rook-interpretation of the numbers ${{\omega}\brack {k}}$.

\smallskip
\begin{cor} \label{cor:rook-xD-Stirling-1st} For any Dyck word $\omega\in\C_n$ in the Weyl algebra $W$, we have
$|\R(\omega,k)|={{\omega}\brack {k}}$.
\end{cor}

\subsection{$q$-Lah number for words} Next, for any word $\omega$ with $n$ $x$'s and $n$ $D$'s, starting with an $x$, notice that the bottom row of $B_{\omega}$  has $n$ cells since the path $\omega$ starts with a north step. Let $B^{*}_{\omega}$ be the board obtained from $B_{\omega}$ by removing the bottom row. Let $\U(\omega,k)$ be the collection of $k$-rook placements of $B^{*}_{\omega}$. For such a rook placement $\sigma\in \U(\omega,k)$, we define the statistic $\inv'(\sigma)$ to be the number of cells in $B^{*}_{\omega}$ that either have
a rook on them, or do not have a rook above them on the same column or to the left of them in the same row. For example, the rook placement shown in Figure \ref{fig:Lah-board} is a member $\sigma\in\U(\omega,2)$ with $\inv'(\sigma)=6$, where the board $B^{*}_{\omega}$ is associated with the word $\omega=xxDxxDDD$.

\begin{figure}[ht]
\begin{center}
\includegraphics[width=1.6in]{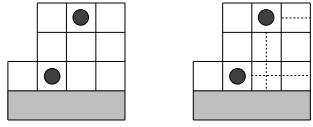}
\end{center}
\caption{\small A rook placement $\sigma\in\U(\omega,2)$ with $\inv'(\sigma)=6$ for $\omega=xxDxxDDD$.}
\label{fig:Lah-board}
\end{figure}

Let $u_q(\omega,k)$ denote the $q$-polynomial of $\U(\omega,k)$ defined as
\[
u_q(\omega,k)=\sum_{\sigma\in\U(\omega,k)} q^{-\inv'(\sigma)}.
\]

\begin{thm} \label{thm:Ferrers-meaning-q-Lah-number} For any word $\omega$ with $n$ $x$'s and $n$ $D$'s in the $q$-deformed Weyl algebra $W$, starting with an $x$, we have
$u_q(\omega,k)={{\omega}\bangle {k}}_q$, i.e.,
\[
\omega=\sum_{k=0}^{n} (-1)^{n-k} u_q(\omega,n-k) xD^{k}x^{k-1}.
\]
\end{thm}

\begin{proof}  The proof is similar to the proof of Theorem \ref{thm:Ferrers-meaning-q-Stirling-1st}. Consider the expansion of $\omega$ over the sequence $\{xD^{k}x^{k-1}\}_{k\ge 0}$ in Eq.\,(\ref{eq:q-word-Lah-number}), the coefficient  ${{\omega}\bangle {k}}_q$ is the number of ways to obtain the word $xD^{k}x^{k-1}$ from $\omega$, by successively substituting $q^{-1}(Dx-1)$ for the $xD$'s other than the prefix peak (i.e., in the beginning of a word). In terms of Ferrers boards, the coefficient ${{\omega}\bangle {k}}_q$ is the number of weighted reductions of the board $B_{\omega}$ to $B_{\mu}$, where $\mu=xD^{k}x^{k-1}$. There are $n-k$ $xD$'s replaced by $-1$. Note that the bottom row of $B_{\mu}$ contains $k$ cells. So the deleted $n-k$ cells amount to $n-k$ non-attacking rooks of the board $B^{*}_{\omega}$, and all replacements take place at the cells that either have a rook on them, or do not have a rook above them on the same column or to the left of them in the same row. Each of these cells is assigned the weight of $q^{-1}$.
Hence the coefficient  ${{\omega}\bangle {k}}_q$ is the weight distribution of the $(n-k)$-rook placements of $B^{*}_{\omega}$, which is exactly the $q$-polynomial  $u_q(\omega,n-k)$ of $\U(\omega,n-k)$.
\end{proof}

\begin{exa} {\rm For $\omega=xxDxxDDD$, we have the expansion
\begin{align*}
\omega=\sum_{k=0}^4 (-1)^{4-k}{{\omega}\bangle {k}}_q xD^{k}x^{k-1} &=-(\frac{1}{q^{7}}+\frac{3}{q^{6}}+\frac{4}{q^{5}}+\frac{3}{q^{4}}+\frac{1}{q^{3}})xD \\
&\qquad +(\frac{2}{q^{4}}+\frac{5}{q^{5}}+\frac{7}{q^{6}}+\frac{6}{q^{7}}+\frac{3}{q^{8}}+\frac{1}{q^{9}})xD^2x \\
&\qquad -(\frac{1}{q^{10}}+\frac{2}{q^{9}}+\frac{3}{q^{8}}+\frac{3}{q^{7}}+\frac{1}{q^{6}})xD^3x^2-\frac{1}{q^{10}}xD^4x^3.
\end{align*}
Note that the polynomial $-u_q(\omega,1)$ coincides with the coefficient of $xD^3x^2$. The ten members in $\U(\omega,1)$ and their weights in $u_q(\omega,1)$ are shown in Figure \ref{fig:Lah-rooks}.
}
\end{exa}

\begin{figure}[ht]
\begin{center}
\psfrag{-6}[][][0.85]{$q^{-6}$}
\psfrag{-7}[][][0.85]{$q^{-7}$}
\psfrag{-8}[][][0.85]{$q^{-8}$}
\psfrag{-9}[][][0.85]{$q^{-9}$}
\psfrag{-10}[][][0.85]{$q^{-10}$}
\includegraphics[width=4.0in]{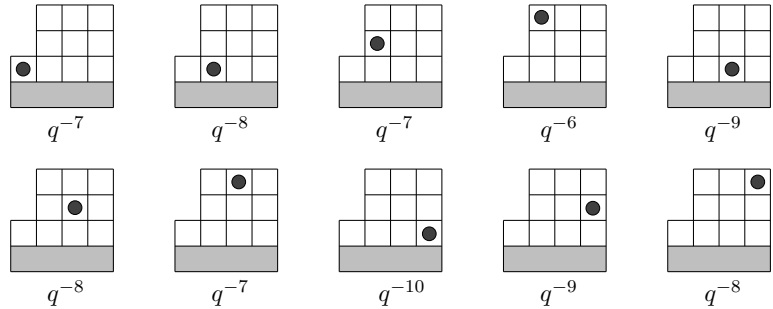}
\end{center}
\caption{\small The members in $\U(\omega,1)$ and their weights in $u_q(\omega,1)$ for $\omega=xxDxxDDD$.}
\label{fig:Lah-rooks}
\end{figure}

Setting $q=1$ in Theorem \ref{thm:Ferrers-meaning-q-Lah-number}, we get a rook-interpretation of the numbers ${{\omega}\bangle {k}}$.

\smallskip
\begin{cor} \label{cor:rook-xD-Lah} For any word $\omega$ with $n$ $x$'s and $n$ $D$'s in the Weyl algebra $W$, starting with an $x$, we have
$|\U(\omega,k)|={{\omega}\bangle {k}}$.
\end{cor}

\section{Enumeration by Rook Factorization Theorem} In this section, we evaluate the numbers ${{\omega}\brack {k}}$ and ${{\omega}\bangle {k}}$ for words $\omega$, making use of the rook-placement enumerations obtained in the previous section.

For a word $\omega\in W$ with $n$ $x$'s and $n$ $D$'s, let $c_i$ be the number of cells in the $i$th column of $B_{\omega}$ for $1\le i\le n$. Goldman et al. \cite{GJW} obtained a Rook Factorization Theorem for a rook polynomial in falling factorials to be completely factorized into linear factors involving the column-heights (see also \cite{Varvak}). The $k$th \emph{falling factorial} of $z$ is
\[
z^{\underline{k}}=z(z-1)\cdots (z-k+1).
\]

\begin{thm}{\rm\bf (Goldman-Joichi-White)}  For a Ferrers Board $B$ with column-heights $c_1,\dots,c_n$,
\[
\sum_{k=0}^{n} r_k(B^{\perp})z^{\underline{n-k}}=\prod_{i=1}^n (z-c_i+i),
\]
where $B^{\perp}$ is the complement of $B$ within the $n\times n$ square and $r_k(B^{\perp})$ is the number of  $k$-rook placements on $B^{\perp}$.
\end{thm}
We obtain analogous results for the rook placements in $\R(\omega,k)$ and in $\U(\omega,k)$, respectively.

\subsection{Evaluating $xD$-Stirling number of the 1st kind by rook factorizations}
We study the evaluation of the numbers ${{\omega}\brack {k}}$ for Dyck words $\omega\in\C_n$. We find that their signed generating function $\sum_{k=0}^n (-1)^{n-k} {{\omega}\brack {k}}z^k$ can be linearly factorized, involving the column-heights of the Ferrers board $B_{\omega}$. We prove the following rook-factorization result, making use of Varvak's method in \cite[Theorem 4.1]{Varvak}.

\begin{thm} \label{thm:rook-factorization-Stirling-1st} For any Dyck word $\omega\in\C_n$ in the Weyl algebra $W$ with the associated Ferrers board $B_{\omega}$, let $c_i$ be the number of cells in the $i$th column of $B_{\omega}$ for $1\le i\le n$. Then
\[
\sum_{k=0}^n (-1)^{n-k} {{\omega}\brack {k}}z^k=\prod_{i=1}^n (z-c_i+i).
\]
\end{thm}

\begin{proof}
Consider the symbols $x,D$ of the Weyl algebra $W$ as the differential operators applied to polynomials $f(t)=t^z$, where $x$ acts as multiplication by $t$, $D=\frac{d}{dt}$ and $z$ is a real number. Note that $(xD)t^z=(t\frac{d}{dt}) t^z=zt^{z}$ and hence $(xD)^{k}t^z=(xD)^{k-1}zt^{z}=\cdots=z^{k}t^z$.   By Eq.\,(\ref{eq:word-1st-Stirling}), we have
\begin{align*}
\omega(t^z)  &= \sum_{k=0}^n (-1)^{n-k}{{\omega}\brack k} (xD)^{k} t^z \\
        &= \sum_{k=0}^n (-1)^{n-k}{{\omega}\brack k} z^k t^z.
\end{align*}
On the left-hand side, $\omega(t^z)$, the application of the $j$th $D$ (from the left) of $\omega$ to $t^z$ gives the linear factor $(z+b_x-b_D)$, where $b_x$ (resp. $b_D$) is the number of times $x$ (resp. $D$) was previously applied, i.e., on the right of the $j$th $D$. Since there are $n-c_j$ $x$'s and $n-j$ $D$'s to the right of the $j$th $D$, we have $b_x=n-c_j$ and $b_D=n-j$. Hence
\[
\omega(t^z)=\left(\prod_{j=1}^{n} (z-c_j+j) \right) t^z.
\]
Setting $t=1$, the assertion follows.
\end{proof}

With this rook-factorization, we remark that the $xD$-Stirling number ${{\omega}\brack {k}}$ can be evaluated in terms of elementary symmetric function. The $k$th \emph{elementary symmetric polynomial} over variables $\{x_1,x_2,\dots,x_n\}$ is
\begin{align*}
e_k(x_1,x_2,\dots,x_n) &= \sum_{i_1<\cdots< i_k} x_{i_1}x_{i_2}\cdots x_{i_k} \quad\mbox{\rm ($k\ge 1$),}\\
e_0(x_1,x_2,\dots,x_n) &= 1.
\end{align*}
It is an elementary fact that
\begin{equation} \label{eq:elementary-symmetric-function}
(z-x_1)(z-x_2)\cdots (x-x_n)=\sum_{k=0}^{n} (-1)^k e_k(x_1,x_2,\dots,x_n) z^{n-k}.
\end{equation}
By Theorem \ref{thm:rook-factorization-Stirling-1st},  we evaluate the number ${{\omega}\brack {k}}$, making use of the column-heights of the Ferrers board $B_{\omega}$, as follows.

\begin{cor} For any Dyck word $\omega\in\C_n$ in the Weyl algebra $W$ with the associated Ferrers board $B_{\omega}$, let $c_i$ be the number of cells in the $i$th column of $B_{\omega}$ for $1\le i\le n$. Then we have
\[
{{\omega}\brack {k}}= e_{n-k}(c_1-1,c_2-2,\dots,c_n-n).
\]
\end{cor}

\subsection{Evaluating $xD$-Lah numbers by rook factorizations}

We define the $k$th \emph{rising factorial} of $z$ by
\[
z^{\overline{k}}=z(z+1)\cdots (z+k-1).
\]

\begin{thm} \label{thm:rook-factorization-Lah-number} For any word $\omega$ with $n$ $x$'s and $n$ $D$'s in the Weyl algebra $W$, starting with an $x$, let $c_i$ be the number of cells in the $i$th column of $B_{\omega}$ for $1\le i\le n$. Then
\[
\sum_{k=0}^n (-1)^{n-k} {{\omega}\bangle {k}} z^{\overline{k}}=\prod_{i=1}^n(z-c_i+i).
\]
\end{thm}

\begin{proof} The proof is similar to the proof of Theorem \ref{thm:rook-factorization-Stirling-1st}. Note that $(\frac{d}{dt})^{k} t^z=z(z-1)\cdots (z-k+1) t^{z-k}$.  By Eq.\,(\ref{eq:word-Lah-number}), we have
\begin{align*}
\omega (t^{z}) &=\sum_{k=0}^n (-1)^{n-k} {{\omega}\bangle {k}} xD^{k}x^{k-1} t^{z} \\
   &=\sum_{k=0}^n (-1)^{n-k} {{\omega}\bangle {k}} xD^{k} t^{z+k-1} \\
   &=\sum_{k=0}^n (-1)^{n-k} {{\omega}\bangle {k}} z^{\overline{k}} t^{z}.
\end{align*}
On the left-hand side, by the same argument as in the proof of Theorem \ref{thm:rook-factorization-Stirling-1st}, we have
\[
\omega(t^z)=\left(\prod_{j=1}^{n} (z-c_j+j) \right) t^z.
\]
Setting $t=1$, the assertion follows.
\end{proof}

For a polynomial in rising factorials $P(z)=\sum_{k=0}^{n} p_k z^{\overline{k}}$, one can check that $p_k=\frac{1}{k!}\Delta^{k} P(-k)$, where $\Delta$ is the difference operator defined by $\Delta P(z)=P(z+1)-P(z)$. In fact, it is known \cite[Eq.\,(1.97)]{EC1} that
\[
p_k=\sum_{i=0}^{k} (-1)^{k-i}{{k}\choose {i}}P(i-k).
\]

By Theorem \ref{thm:rook-factorization-Lah-number},  we evaluate the number ${{\omega}\bangle {k}}$, making use of the column-heights of the Ferrers board $B_{\omega}$, as follows.

\begin{cor} For any word $\omega$ with $n$ $x$'s and $n$ $D$'s in the Weyl algebra $W$, starting with an $x$, let $P(z)=\prod_{j=1}^{n} (z-c_j+j)$, where $c_i$ is the number of cells in the $i$th column of $B_{\omega}$ for $1\le i\le n$. Then
\[
{{\omega}\bangle {k}}= \frac{1}{k!} \sum_{i=0}^{k} (-1)^{n-i}{{k}\choose {i}}P(i-k).
\]
\end{cor}

\subsection{$q$-analogues of rook factorization results} Recall that the $q$-analogue of positive integer $n$ is $[n]_q=1+q+\cdots+q^{n-1}$. The commutation relation $Dx-qxD=1$ of the $q$-deformed Weyl algebra is realized by the $q$-analogue of the derivative $D=D_q=\frac{d}{dt}$ acting on polynomials $f(t)$ by
\[
(D_q f)(t) := \frac{f(qt)-f(t)}{(q-1)t},
\]
and the operator $x$ acting by multiplication by $t$. Note that $D_q(t^n)=[n]_q t^{n-1}$.
Analogous to Theorem \ref{thm:rook-factorization-Stirling-1st} and Theorem \ref{thm:rook-factorization-Lah-number}, we have the following variations of the $q$-rook Factorization Theorem of Garsia and Remmel \cite{GR}.

\begin{thm} \label{thm:rook-factorization-q-Stirling-1st} Given a Dyck word $\omega\in\C_n$ in the $q$-deformed Weyl algebra $W$ with the associated Ferrers board $B_{\omega}$, let $c_i$ be the number of cells in the $i$th column of $B_{\omega}$ for $1\le i\le n$. Then
\[
\sum_{k=0}^n (-1)^{n-k} {{\omega}\brack {k}}_q [z]_q^k=\prod_{i=1}^n [z-c_i+i]_q.
\]
\end{thm}

\begin{thm} \label{thm:rook-factorization-q-Lah-number} For any  word $\omega$ with $n$ $x$'s and $n$ $D$'s in the $q$-deformed Weyl algebra $W$, starting with an $x$, let $c_i$ be the number of cells in the $i$th column of $B_{\omega}$ for $1\le i\le n$. Then
\[
\sum_{k=0}^n (-1)^{n-k} {{\omega}\bangle {k}}_q [z]_q[z+1]_q\cdots [z+k-1]_q=\prod_{i=1}^n [z-c_i+i]_q.
\]
\end{thm}

Making use of the derivative $D_q(t^z)=[z]_q t^{z-1}$, the above two theorems can be proved by the same arguments as in the proofs of Theorem \ref{thm:rook-factorization-Stirling-1st} and Theorem \ref{thm:rook-factorization-Lah-number}.

\section{Evaluating $xD$-Stirling numbers by chromatic polynomials}
For a Dyck path $\omega\in\C_n$, an east step of $\omega$ is said to be at \emph{height} $j$ if the east step goes from the line $y=x+j+1$ to the line $y=x+j$. Let $h_i$ be the height of the $i$th east step of $\omega$ for $1\le i\le n$. Note that $h_i=c_i-i$, where $c_i$ is the height of the $i$th column of the Ferrers board $B_{\omega}$.   The \emph{chromatic polynomial} of a simple graph $G$, denoted as $\chi_G(z)$, is the number of proper vertex-coloring
of $G$ using $z$ colors.  By the construction of the quasi-threshold graph $G_{\omega}$ associated with $\omega$, the $j$th east step is associated with the vertex $j$, which is adjacent to $h_j$ vertices of $\{j+1,\dots,n\}$ in $G_{\omega}$.
So if we color the vertices of $G_{\omega}$ in reverse order, using $z$ colors, by the first-fit algorithm then the chromatic polynomial of $G_{\omega}$ is
\begin{equation} \label{eq:chromatic-poly-Dyck-word}
\chi_{G_{\omega}}(z)=\prod_{j=1}^{n} (z-h_j).
\end{equation}
This proves a result of Engbers et al. \cite[Claim 3.3]{EGH}.

As a consequence of Eq.\,(\ref{eq:chromatic-poly-Dyck-word}),
along with the rook factorization results in Theorems \ref{thm:rook-factorization-Stirling-1st} and \ref{thm:rook-factorization-Lah-number}, we have the following result.

\begin{cor}  \label{cor:rook-factor-chromatic-poly}  Given a Dyck word $\omega\in\C_n$ with the associated quasi-threshold graph $G_{\omega}$, we have
\[
\chi_{G_{\omega}}(z)=\sum_{k=0}^n (-1)^{n-k} {{\omega}\brack {k}}z^k=\sum_{k=0}^n (-1)^{n-k} {{\omega}\bangle {k}} z^{\overline{k}}.
\]
\end{cor}

By Whitney's theorem \cite{Whitney}, the first identity in the above corollary provides another combinatorial interpretation of ${{\omega}\brack {k}}$ in terms of the subgraphs of $G_{\omega}$ without broken circuits.

Given a simple graph $G$ with $n$ vertices and a totally ordered edge set $(E(G),<)$, a \emph{broken circuit} of $G$ is a subgraph obtained from removing from some circuit in $G$ the greatest edge.

\begin{thm} {\rm\bf (Whitney)} Let $d_j$ be the number of subgraphs consisting of $j$ edges of $G$  without broken circuits. Then the chromatic polynomial $\chi_G(z)$ of $G$ is
\begin{equation}
\chi_{G}(z) =\sum_{k=0}^{n} (-1)^{n-k} d_{n-k} z^k.
\end{equation}
\end{thm}

By Eq.\,(\ref{eq:chromatic-poly-Dyck-word}) and Corollary \ref{cor:rook-factor-chromatic-poly}, for a Dyck word $\omega\in\C_n$, it follows from Whitney's theorem that ${{\omega}\brack {k}}$ counts the number of subgraphs consisting of $n-k$ edges of $G_{\omega}$ without broken circuits. In the following, we present an immediate bijection between two families of subgraphs of $G_{\omega}$ enumerated by the number ${{\omega}\brack {k}}$.

\begin{thm} \label{thm:broken-circuit-bijection} Given a Dyck word $\omega\in\C_n$ with the associated quasi-threshold graph $G_{\omega}$, there is a bijection between the set of partitions of $G_{\omega}$ into $k$-component decreasing forests and the set
of subgraphs consisting of $n-k$ edges of $G_{\omega}$ without broken circuits.
\end{thm}

\begin{proof} With the vertex set $\{1,\dots,n\}$ of $G_{\omega}$,  we denote the edge connecting two adjacent vertices $i,j$ ($j>i$) by the ordered pair $(j,i)$. Then we assign a total order on the edge set of $G_{\omega}$ by the lexicographical order of the ordered pairs, i.e., two edges $(x_1,x_2)<(y_1,y_2)$ if $x_i< y_i$ for the first $i$ where $x_i$ and $y_i$ differ.

On the basis of the edge-ordering, we observe that every $k$-component decreasing forest $\alpha$ of $G_{\omega}$ is exactly a subgraph consisting of $n-k$ edges without broken circuits. If not, $\alpha$ contains a broken circuit $\beta=x_1,x_2,\dots,x_t$ with the missing edge $(x_1,x_t)$, then $x_1,x_t$ are the greatest two vertices in $\beta$, which implies that there is a vertex $v_j$ ($1<j<t$) such that $x_1>x_j$ and $x_t>x_j$. This contradicts that $\beta$ is a decreasing path from $x_1$ to $x_t$. The assertion follows.
\end{proof}

\section{Concluding Remarks}
For a computational purpose, it is desirable for the words $\omega\in W$ to have the $x$'s completely to the right and the $D$'s completely to the left. See \cite{BHPSD} for information. That makes the normal order coefficients of $\omega$ have the characteristics of the Stirling numbers of second kind.  For combinatorial interest, we study the companion expansions with coefficients being generalizations of the Stirling numbers of the first kind and the Lah numbers, as intermediate stages of the normally ordered forms. There are other models for the normal order problem. For example, the \emph{gate diagrams} introduced by Blasiak and Flajolet \cite{BF} and \emph{path decompositions of digraphs} used by Dzhumadil'daev and Yeliussizov \cite{DY}. A large portion of existing results focused on the combinatorial interpretations of the normal order coefficients. We are interested in the interpretations of the numbers ${{\omega}\brack {k}}$ and ${{\omega}\bangle {k}}$ in the combinatorial models.

\end{document}